\documentclass{svmult}

\renewcommand{\email}[1]{\emailname: #1} 

\usepackage{mathptmx}       
\usepackage{helvet}         
\usepackage{courier}        

\usepackage{makeidx}         
\usepackage{graphicx}        
\usepackage[bottom]{footmisc}

\usepackage{latexsym}
\usepackage{amsmath}
\usepackage{amsfonts}
\usepackage{amssymb}
\usepackage{bm}

\usepackage{url}
\usepackage[misc,geometry]{ifsym}

\renewenvironment{proof}{\noindent{\itshape Proof.}}{\smartqed\qed}

\spdefaulttheorem{assumption}{Assumption}{\upshape \bfseries}{\itshape}
\spdefaulttheorem{algo}{Algorithm}{\upshape \bfseries}{\itshape}

\newcommand{\N}{{\mathbb{N}}} 

\DeclareSymbolFont{bbold}{U}{bbold}{m}{n}
\DeclareSymbolFontAlphabet{\mathbbold}{bbold}

\begin{document}

\title*{Analysis of Framelet Transforms on a Simplex}

\author{Yu Guang Wang \and Houying Zhu}

\institute{
 Yu Guang Wang (\Letter)
 \at The University of New South Wales, Sydney, Australia \\
 \email{yuguang.wang@unsw.edu.au}
 \and
 Houying Zhu
 \at The University of Melbourne, Melbourne, Australia \\
 \email{houying.zhu@unimelb.edu.au}
}

\maketitle

\index{Wang, Y. G.}
\index{Zhu, H.}

\paragraph{Dedicated to Ian H. Sloan on the occasion of his 80th birthday with\\ our gratitude for his constant supervision, support and encouragement.}

\abstract{
In this paper, we construct framelets associated with a sequence of quadrature rules on the simplex $T^{2}$ in $\mathbb{R}^{2}$. We give the framelet transforms --- decomposition and reconstruction of the coefficients for framelets of a function on $T^{2}$. We prove that the reconstruction is exact when the framelets are tight. We give an example of construction of framelets and show that the framelet transforms can be computed as fast as FFT.
}

\section{Introduction}
\label{sec:intro}
Multiresolution analysis on a simplex $T^{2}$ in $\mathbb{R}^{2}$ has many applications such as in numerical solution of PDEs and computer graphics \cite{Dyn1992,deFeDeTo2016,GrCoMaDe2017}.
In this paper, we construct framelets (or a framelet system) on $T^{2}$, following the framework of \cite{WaZh2017}, and give the transforms of coefficients for framelets.

Framelets are localised functions associated with quadrature rules of $T^{2}$. 
Each framelet is scaled at a level $j$, $j=0,1,\dots$ and translated at a node of a quadrature rule of level $j$. The \emph{framelet coefficients} for a square-integrable function $f$ on the simplex are the inner products of the framelets with $f$ on $T^{2}$.
We give the \emph{framelet transforms} which include the \emph{decomposition} and \emph{reconstruction} of the coefficients for framelets. Since the framelets are well-localised, see e.g. \cite{MaMh2008}, the decomposition gives all \emph{approximate} and \emph{detailed} information of the function $f$. This plays an important role in signal processing on the simplex.

For levels $j$ and $j+1$, the decomposition estimates the framelet coefficients of level $j+1$ by the coefficients of level $j$. The reconstruction is the inverse, which estimates the coefficients of level $j$ by the level $j+1$.
Such framelet transforms are significant as by decompositions or reconstructions, we are able to estimate high-level framelet coefficients from the bottom level $0$, or the inverse.

We show that when the quadrature rules and masks have good properties, the reconstruction is exact and invertible with the decomposition, see Section~\ref{sec:reconstruction}.
We also show that the framelet transforms can be computed as fast as the FFTs, see Section~\ref{sec:fast.comput}.

We construct framelets using tensor-product form of Jacobi polynomials and triangular Kronecker lattices \cite{BaOw2015} with equal weights, see Section~\ref{sec:constr.example}.

\section{Framelets on Simplex}

In the paper, we consider the \emph{simplex} (or the \emph{triangle}) 
\begin{equation*}
	T^{2}:=\{\boldsymbol{x}:=(x_{1},x_{2})| x_{1}\ge0, x_{2}\ge0, x_{1}+x_{2}\le1\}.
\end{equation*}
Let $L_{2}(T^2)$ be the space of complex-valued square integrable functions on $T^{2}$ with respect to the normalized Lebesgue area measure $\mu$ on $\mathbb{R}^{2}$ (i.e. $\int_{T^{2}}\mathrm{d}{\mu(\boldsymbol{x})}=1$), provided with inner product
$\langle f,g \rangle:=\langle f,g \rangle_{L_{2}(T^2)}:=\int_{T^{2}}f(\boldsymbol{x})\mathrm{d}{\mu(\boldsymbol{x})}$ 
, where $\overline{g}$ is the complex conjugate of $g$, and endowed with
 the induced $L_{2}$-norm $\| f \|_{L_{2}(T^2)}:=\sqrt{\langle{f,f}\rangle}$ for $f\in L_{2}(T^2)$.

For $\ell\ge0$, let $\mathcal{V}_{\ell}:=\mathcal{V}_{\ell}(T^{2})$ be the space of orthogonal polynomials of degree $\ell$ with respect to the inner product $\langle {\cdot,\cdot}\rangle_{L_{2}(T^2)}$. The dimension of $\mathcal{V}_{\ell}$ is $\ell+1$, see \cite{DuXu2014}. The elements of $\mathcal{V}_{\ell}$ are said to be the \emph{polynomials} of degree $\ell$ on $T^{2}$. The union of all polynomial spaces $\cup_{\ell=0}^{\infty}\mathcal{V}_{\ell}$ is dense in $L_{2}(T^2)$. 

As a compact Riemannian manifold, the simplex $T^{2}$ has the Laplace-Beltrami operator
\begin{equation*}
  \Delta:=\sum_{i=1}^{2}x_{i}(1-x_{i})\frac{\partial^{2}}{\partial 
  {x}_{i}^{2}} - 2 \sum_{1\le i\le j\le2}x_{i}x_{j}\frac{\partial^{2}}{\partial x_{i} \partial x_{j}} + \sum_{i=1}^{2}(1-3 x_{i})\frac{\partial}{\partial {x}_i},
\end{equation*}
with polynomials $P_{\ell}$ in $\mathcal{V}_{\ell}$ as the eigenfunctions and with (square-rooted) eigenvalues $\lambda_{\ell}:=\sqrt{\ell(\ell+2)}$:
\begin{equation*}\label{eq:LBm.simplex}
  -\Delta P_{\ell} = \lambda_{\ell}^{2}\: P_{\ell},\quad \ell\in\mathbb{N}_{0},
\end{equation*}
where $\mathbb{N}_{0}:=\{0,1,2,\dots\}$.

Let $L_{1}(\mathbb{R})$ be the space of absolutely integrable functions on $\mathbb{R}$ with respect to the Lebesgue measure and let $l_{1}(\mathbb{Z})$ be the set of $l_{1}$ summable sequences on $\mathbb{Z}$. For $r\ge1$, let $\Psi:=\{\alpha; \beta^1,\ldots,\beta^r\}$ be a set of $(r+1)$ functions in $L_{1}(\mathbb{R})$, which are associated with a \emph{filter bank} ${\boldsymbol{\eta}}:=\{a; b_{1},\ldots,b_{r}\}\subset l_1(\mathbb{Z})$ satisfying
\begin{equation}
\label{eq:refinement}
    \widehat{\alpha}(2\xi) = \widehat{a}(\xi)\widehat{\alpha}(\xi),\quad
    \widehat{\beta^n}(2\xi) = \widehat{b_{n}}(\xi)\widehat{\alpha}(\xi),\quad n=1,\ldots,r, \; \xi\in\mathbb{R},
\end{equation}
where $\widehat{g}(\xi):=\int_{\mathbb{R}}g(x)e^{-2\pi i x\xi} \:\mathrm{d}{x}$, $\xi\in\mathbb{R}$ is the Fourier transform for $g\in L_{1}(\mathbb{R})$ and $\widehat{h}(\xi):=\sum_{k\in\mathbb{Z}}h_{k}e^{-2\pi {\mathrm{i}} \xi}$ is the Fourier series of a sequence $h:=(h_k)_{k\in\mathbb{Z}}$ in $l_{1}(\mathbb{Z})$. Here, the sequences $a$ and $b_{n}$ are said to be \emph{low-pass (mask)} and \emph{high-pass (mask)} respectively. 

We introduce the continuous and semi-discrete framelets on the simplex following the construction and notation of \cite{WaZh2017,Dong2015}. The \emph{continuous framelets} on the simplex $T^{2}$ are, for $j\in\mathbb{N}_{0}$,
\begin{equation}\label{eq:intro.cfr}
\begin{array}{l}
    \displaystyle\boldsymbol{\varphi}_{j,\boldsymbol{y}}(\boldsymbol{x}) :=  \sum_{\ell=0}^{\infty}\widehat{\alpha}\left(\frac{\lambda_{\ell}}{2^{j}}\right)\overline{P_{\ell}(\boldsymbol{y})}P_{\ell}(\boldsymbol{x}),\\
    \displaystyle\boldsymbol{\psi}_{j,\boldsymbol{y}}^{n}(\boldsymbol{x}) :=  \sum_{\ell=0}^{\infty}\widehat{\beta^n}\left(\frac{\lambda_{\ell}}{2^{j}}\right)\overline{P_{\ell}(\boldsymbol{y})}P_{\ell}(\boldsymbol{x}),\quad n = 1,\ldots,r.
    \end{array}
\end{equation}

The continuous framelets in \eqref{eq:intro.cfr} are analogues of continuous wavelets in $\mathbb{R}$. The level ``$j$'' indicates the ``dilation'' scale and ``$\boldsymbol{y}$'' is the point at which the framelet is ``translated''.

Let ${Q}_{N}:=\{(w_{j},\boldsymbol{x}_{j})\}_{j=1}^{N}$, which is a set of $N$ pairs of weights $w_{j}$ in $\mathbb{R}\backslash\{0\}$ and points $\boldsymbol{x}_{j}$ on $T^{2}$, define the quadrature rule 
\begin{equation*}
	{Q}_{N} [f]:= \sum_{k=1}^{N} w_j f(\boldsymbol{x}_j)
\end{equation*}
for continuous functions $f$ on $T^{2}$.
Let ${Q}_{N_{j}}:=\{(\omega_{j,k},\boldsymbol{x}_{j,k})\}_{k=1}^{N_{j}}$, $j\in\mathbb{N}_{0}$, be a sequence of such quadrature rules.
For $j=0,1,\dots$, the \emph{semi-discrete framelets} $\boldsymbol{\varphi}_{j,k}$ and $\boldsymbol{\psi}_{j,k'}^{n}$ associated with quadrature rules ${Q}_{N_{j}}$ are defined as the continuous framelets $\boldsymbol{\varphi}_{j,\boldsymbol{y}}$ and $\boldsymbol{\psi}_{j,\boldsymbol{y}}^{n}$ translated at $\boldsymbol{x}_{j,k}$ and $\boldsymbol{x}_{j+1,k'}$ respectively. That is, for $k=1,\dots,N_{j}$,
\begin{equation}\label{eq:fra}
    \boldsymbol{\varphi}_{j,k}(\boldsymbol{x}) := \sqrt{\omega_{j,k}}\:\boldsymbol{\varphi}_{j,\boldsymbol{x}_{j,k}}(\boldsymbol{x}) = \sqrt{\omega_{j,k}}\sum_{\ell=0}^{\infty}\widehat{\alpha}\left(\frac{\lambda_{\ell}}{2^{j}}\right)\overline{P_{\ell}(\boldsymbol{x}_{j,k})}P_{\ell}(\boldsymbol{x}),
\end{equation}
and for $k'=1,\dots,N_{j+1}$ and $n =1,\ldots,r$,
\begin{equation}\label{eq:frb}
    \boldsymbol{\psi}_{j,k'}^{n}(\boldsymbol{x}) := \sqrt{{\omega_{j+1,k'}}}\:\boldsymbol{\psi}_{j,{\boldsymbol{x}_{j+1,k'}}}^{n}(\boldsymbol{x}) = \sqrt{\omega_{j+1,k'}}\sum_{\ell=0}^{\infty} \widehat{\beta^{n}}\left(\frac{\lambda_{\ell}}{2^{j}}\right)\overline{P_{\ell}(\boldsymbol{x}_{j+1,k'})}P_{\ell}(\boldsymbol{x}).
\end{equation}
We say $\boldsymbol{\varphi}_{j,k}$ and $\boldsymbol{\psi}_{j,k'}^{n}$ are \emph{low-pass framelet} and \emph{high-pass framelet} respectively.

Note that here we use the ${Q}_{N_{j+1}}$ for high-passes due to the scale of $\boldsymbol{\psi}_{j,k'}^{n}$ is at $j+1$. This will be clear in Section~\ref{sec:constr.example}.

We also use the notation $\boldsymbol{\psi}_{j,k}^{n}$ for $\boldsymbol{\psi}_{j,k'}^{n}$ if no confusion arises.

The framelets $\boldsymbol{\varphi}_{j,k}$ and $\boldsymbol{\psi}_{j,k}^{n}$ corresponding to the low-pass $a$ and high-pass $b_{n}$ carry the information of approximations and details in framelet transforms, as we will show below.

\section{Decomposition for Framelets}\label{sec:decomposition}
In practice, we need to estimate the framelet coefficients of high levels from low-level coefficients. This can be achieved by the \emph{decomposition} of framelets.

The decomposition for framelets can be realized by the operations of convolution and downsampling as we introduce now.

Let $h\in l_1(\mathbb{Z})$ be a mask satisfying that the support of the Fourier series $\widehat{h}$ of $h$ is a subset of $[0,1/2]$. Let $l(N)$ be the set of complex-valued sequences with supports in $[0,N]$. Let ${Q}_{N_{j}}:=\{(\omega_{j,k},\boldsymbol{x}_{j,k})\}_{k=1}^{N_j}$, $j\in\mathbb{N}_{0}$, be the quadrature rules for framelets. Let $l({Q}_{N_{j}})$ be the set of sequences $\mathrm{v}$ in $l(N_j)$ satisfying that there exists a sequence $\mathrm{u}$ in $\l_1(\mathbb{Z})$ such that 
\begin{equation*}
	(\mathrm{v})_{k} = \sqrt{\omega_{j,k}}\sum_{\ell=0}^{\infty}\mathrm{u}_{\ell}\: P_{\ell}(\boldsymbol{x}_{j,k}),\quad k=1,\dots,N_j.
\end{equation*}
We let $\widehat{\mathrm{v}}_{\ell}:=\mathrm{u}_{\ell}$ (with abuse of notation) be the (generalized) \emph{Fourier coefficients} of $\mathrm{v}$ for the orthonormal basis $P_{\ell}$ and the quadrature rule ${Q}_{N_{j}}$ on $T^{2}$.

The \emph{(discrete) convolution}  $\mathrm{v} \ast_{j} h$ of a sequence $\mathrm{v}$ with the mask $h$ is a sequence in $l({Q}_{N_{j}})$ given by
\begin{equation}\label{eq:dis.conv}
    (\mathrm{v} \ast_{j} h)_{k} := \sum_{\ell=0}^{\infty} \widehat{\mathrm{v}}_{\ell}\: {\widehat{h}}\left(\frac{\lambda_{\ell}}{2^{j}}\right)\sqrt{\omega_{j,k}}\:P_{\ell}(\boldsymbol{x}_{j,k}),\quad k=1,\dots,N_{j}.
\end{equation}
Then, the Fourier coefficients of $\mathrm{v} \ast_{j} h$ are $(\widehat{\mathrm{v} \ast_{j} h})_{\ell} = \widehat{\mathrm{v}}_{\ell}\:{\widehat{h}}\left(\frac{\lambda_{\ell}}{2^{j}}\right)$, $\ell\in\mathbb{N}_{0}$.

The \emph{downsampling} $\mathrm{v}\hspace{-0.8mm}\downarrow_{j}$, $j\ge1$, of a sequence $\mathrm{v}\in l({Q}_{N_{j}})$ is a sequence in $l({Q}_{N_{j-1}})$ given by
\begin{equation}\label{eq:downsample}
    (\mathrm{v}\hspace{-0.8mm}\downarrow_{j})_{k} := \sum_{\lambda_{\ell}\le2^{j-1}} \widehat{\mathrm{v}}_{\ell} \:\sqrt{\omega_{j,k}}\:P_{\ell}(\boldsymbol{x}_{j,k}),\quad k=1,\dots,N_{j-1}.
\end{equation}

For semi-discrete framelets in \eqref{eq:fra} and \eqref{eq:frb}, the inner products $\langle{f,\boldsymbol{\varphi}_{j,k}}\rangle$ and $\langle{f,\boldsymbol{\psi}_{j,k'}^{n}}\rangle$, $n=1,\ldots,r$, $j\in\mathbb{N}_{0}$, $k=1,\dots,N_j$ and $k'=1,\dots,N_{j+1}$, are said to be \emph{framelet coefficients} for $f$.
For convenience, we let $\mathrm{v}_{j}$ and $\mathrm{w}^{n}_{j}$ denote the framelet coefficients for $f$: 
\begin{equation}\label{eq:coe.framelets}
    (\mathrm{v}_{j})_{k} := \langle{f,\boldsymbol{\varphi}_{j,k}}\rangle,\quad (\mathrm{w}^{n}_{j})_{k'} := \langle{f,\boldsymbol{\psi}_{j,k'}^{n}}\rangle.
\end{equation} 

The Fourier coefficients of a function $f\in L_{2}(T^2)$ are $\widehat{f}_{\ell} := \langle{f,P_{\ell}}\rangle$, $\ell\in\mathbb{N}_{0}$. Let $h$ be a mask in $l_1(\mathbb{Z})$ and $h^{\star}$ be the mask whose Fourier series is conjugate to the Fourier series of $h$.

The following proposition shows the decomposition for framelet coefficients between adjacent levels.
\begin{proposition}
\label{prop:decomposition}
Let the framelet coefficients for semi-discrete framelets in \eqref{eq:fra} and \eqref{eq:frb} be given by \eqref{eq:coe.framelets}, where the supports of $\widehat{\alpha}$ and $\widehat{\beta^{n}}$ are subsets of $[0,1/2]$. For $j=1,2,\dots$,
the decomposition from level $j$ into level $j-1$ is  
\begin{equation}
\label{eq:decomp}
  \mathrm{v}_{j-1} = (\mathrm{v}_{j}\ast_{j} a^\star)\hspace{-0.8mm}\downarrow_{j},\quad \mathrm{w}^{n}_{j-1} = \mathrm{v}_{j}\ast_{j} b_{n}^\star,\quad  n=1,\ldots,r.
\end{equation}
\end{proposition}
\begin{proof}
For $f\in L_{2}(T^2)$, by the orthonormality of $P_{\ell}$ and \eqref{eq:coe.framelets},
\begin{equation*}\label{eq:fr.coeff}
\begin{array}{ll}
(\mathrm{v}_{j-1})_{k} &= \displaystyle\sqrt{\omega_{j-1,k}}\sum_{\ell=0}^{\infty} \overline{\widehat{\alpha}\left(\frac{\lambda_{\ell}}{2^{j-1}}\right)}\widehat{f}_{\ell}\:P_{\ell}(\boldsymbol{x}_{j-1,k}),\quad k=1,\dots,N_{j-1},\\[5mm]
(\mathrm{w}^{n}_{j-1})_{k'} &= \displaystyle\sqrt{\omega_{j,k'}}\sum_{\ell=0}^{\infty} \overline{\widehat{\beta^{n}}\left(\frac{\lambda_{\ell}}{2^{j-1}}\right)}\widehat{f}_{\ell}\:P_{\ell}(\boldsymbol{x}_{j,k'}),\quad k'=1,\dots,N_{j},\;n=1,\dots,r.
    \end{array}
\end{equation*}
For low-pass, by \eqref{eq:refinement}, \eqref{eq:dis.conv} and \eqref{eq:downsample}, for $k=1,\ldots, N_{j-1}$,
\begin{align*}
    (\mathrm{v}_{j-1})_{k} &= \sqrt{\omega_{j-1,k}}\sum_{\lambda_{\ell}\le2^{j-1}}\widehat{f}_{\ell}\: \overline{\widehat{\alpha}\left(\frac{\lambda_{\ell}}{2^{j-1}}\right)}P_{\ell}(\boldsymbol{x}_{j-1,k})\notag\\
    &= \sqrt{\omega_{j-1,k}}\sum_{\lambda_{\ell}\le2^{j-1}}\widehat{f}_{\ell}\:\overline{\widehat{\alpha}\left(\frac{\lambda_{\ell}}{2^{j}}\right)}\:\overline{\widehat{a}\left(\frac{\lambda_{\ell}}{2^{j}}\right)} P_{\ell}(\boldsymbol{x}_{j-1,k})\\
    &= \sqrt{\omega_{j-1,k}}\sum_{\lambda_{\ell}\le2^{j-1}}\widehat{(\mathrm{v}_{j})}_{\ell}\:\overline{\widehat{a}\left(\frac{\lambda_{\ell}}{2^{j}}\right)}P_{\ell}(\boldsymbol{x}_{j-1,k})\\
    &= \bigl((\mathrm{v}_{j} \ast_{j} a^\star)\hspace{-0.8mm}\downarrow_{j}\bigr)_{k}.
\end{align*}
For high-passes, for $k'=1,\ldots, N_{j}$ and $n=1,\dots,r$,
\begin{align*}
(\mathrm{w}^{n}_{j-1})_{k'}
&=\sqrt{\omega_{j,k'}}\sum_{\lambda_{\ell}\le 2^{j-1}}\widehat{f}_{\ell}\: \overline{\widehat{\beta^{n}}\left(\frac{\lambda_{\ell}}{2^{j-1}}\right)}P_{\ell}(\boldsymbol{x}_{j,k'})\\
&=\sqrt{\omega_{j,k'}}\sum_{\lambda_{\ell}\le 2^{j-1}}\widehat{f}_{\ell}\: \overline{\widehat{\beta^{n}}\left(\frac{\lambda_{\ell}}{2^{j}}\right)}\:\overline{\widehat{b_{n}}\left(\frac{\lambda_{\ell}}{2^{j}}\right)}P_{\ell}(\boldsymbol{x}_{j,k'})\\
&=\sqrt{\omega_{j,k'}}\sum_{\lambda_{\ell}\le 2^{j-1}}\widehat{(\mathrm{v}_{j})}_{\ell}\:\overline{\widehat{b_{n}}\left(\frac{\lambda_{\ell}}{2^{j}}\right)} P_{\ell}(\boldsymbol{x}_{j,k'})\\
&= (\mathrm{v}_{j} \ast_{j} b_{n}^\star)_{k'}.
\end{align*}
These give \eqref{eq:decomp}.
\end{proof}

\section{Reconstruction for Tight Framelets}\label{sec:reconstruction}

We say the set of framelets $\{\boldsymbol{\varphi}_{j,k},\boldsymbol{\psi}_{j,k'}^{n}|n=1,\dots,r,\;  k=1,\dots,N_{j},\; k'=1,\dots,N_{j+1},\; j\in\mathbb{N}_{0}\}$ a \emph{tight frame} for $L_{2}(T^2)$ if the framelets are all in $L_{2}(T^2)$, and in the $L_2$ sense,
\begin{equation*}\label{eq:intro.fr.tight.f}
    f = \sum_{k=1}^{N_0} \langle{f,\boldsymbol{\varphi}_{j,k}}\rangle\boldsymbol{\varphi}_{j,k}  + \sum_{j=0}^{\infty} \sum_{k'=1}^{N_{j+1}}\sum_{n=1}^{r}\langle{f,\boldsymbol{\psi}_{j,k'}^{n}}\rangle\boldsymbol{\psi}_{j,k'}^{n} \quad \mbox{for all}~ f\in L_{2}(T^2),
\end{equation*}
or equivalently,
\begin{equation*}
\label{thmeq:framelet.tightness}
 \| f \|_{L_{2}(T^2)}^{2} = \sum_{k=1}^{N_{0}} \bigl|\langle{f,\boldsymbol{\varphi}_{j,k}}\rangle \bigr|^{2} + \sum_{j=0}^{\infty} \sum_{k'=1}^{N_{j+1}} \sum_{n=1}^{r}\bigl|\langle{f,\boldsymbol{\psi}_{j,k'}^{n}}\rangle\bigr|^{2}\quad \mbox{for all}~ f\in L_{2}(T^2).
\end{equation*}
The framelets are then said to be \emph{(semi-discrete) tight framelets}.

If the framelets are tight on the simplex, a function in $L_{2}(T^2)$ can be represented using the framelet coefficients. The following property as a consequence of \cite[Theorem~2.4]{WaZh2017} shows that the tightness of framelets is equivalent to a multiscale representation of framelets of a level by lower levels. 
\begin{proposition}[\cite{WaZh2017}]\label{prop:tightness} The semi-discrete framelets in \eqref{eq:fra} and \eqref{eq:frb} are tight if and only if for all $f\in L_{2}(T^2)$, the following identities hold:
\begin{align}
            & \lim_{j\to\infty} \sum_{k=1}^{N_{j}} \bigl|\langle{f,\boldsymbol{\varphi}_{j,k}}\rangle\bigr|^{2}=\| f \|_{L_{2}(T^2)}^{2},\notag\\
            & \sum_{k=1}^{N_{j+1}} \bigl|\langle{f,\boldsymbol{\varphi}_{j+1,k}}\rangle\bigr|^{2}=\sum_{k=1}^{N_{j}} \bigl|\langle{f,\boldsymbol{\varphi}_{j,k}}\rangle\bigr|^{2} + \sum_{k=1}^{N_{j+1}}\sum_{n=1}^{r}\bigl|\langle{f,\boldsymbol{\psi}_{j,k}^{n}}\rangle\bigr|^{2} ,\quad j\in\mathbb{N}_{0}.\label{thmeq:framelet.coe.j.j1}
\end{align}	
\end{proposition}

The condition in \eqref{thmeq:framelet.coe.j.j1} implies that high-level framelet coefficients can be estimated by low levels. This then gives the \emph{reconstruction} for framelets.

The reconstruction depends on the property of the quadrature rules ${Q}_{N_{j}}$ for framelets. A quadrature rule $Q_N:=\{(w_{j},\boldsymbol{x}_{j})\}_{j=1}^{N}$ is said to be exact for polynomials up to degree $\ell$ if for $\ell'=0,\dots,\ell$,
\begin{equation*}
	\int_{T^{2}}p_{\ell'}(\boldsymbol{x})\mathrm{d}{\mu(\boldsymbol{x})} = \sum_{j=1}^{N}w_{j} p_{\ell'}(\boldsymbol{x}_{j})\quad \mbox{for all}~p_{\ell'}\in \mathcal{V}_{\ell'}.
\end{equation*}
When the quadrature rule ${Q}_{N_{j}}$, $j\in\mathbb{N}_{0}$, for framelets $\boldsymbol{\varphi}_{j,k}$ and $\boldsymbol{\psi}_{j,k}^{n}$ is exact for polynomials up to degree $2^j$, the tightness of the framelets is equivalent to the following condition on masks:
\begin{equation}\label{eq:tight.masks}
	\lim_{j\to\infty}\widehat{a}\left(\frac{\lambda_{\ell}}{2^j}\right)=1,\quad \left|\widehat{a}\left(\frac{\lambda_{\ell}}{2^j}\right)\right|^{2} + \sum_{n=1}^{r}\left|\widehat{b_{n}}\left(\frac{\lambda_{\ell}}{2^j}\right)\right|^{2}=1\quad \mbox{for~} j,\ell\in\mathbb{N}_{0},
\end{equation}
see \cite[Theorem~2.1 and Corollary~2.6]{WaZh2017}.

The \emph{upsampling} $\mathrm{v}\hspace{-0.8mm}\uparrow_{j}$, $j\ge1$, of a sequence $\mathrm{v}\in l({Q}_{N_{j-1}})$ is a sequence in $l({Q}_{N_{j}})$ given by
\begin{equation*}
	(\mathrm{v}\hspace{-0.8mm}\uparrow_{j})_{k} :=
\sum_{\lambda_{\ell}\le2^{j-2}}\widehat{\mathrm{v}}_{\ell}\sqrt{\omega_{j,k}}\: P_{\ell}(\boldsymbol{x}_{j,k}),\quad k = 1,\ldots, N_{j},
\end{equation*}
where $\widehat{\mathrm{v}}_{\ell}$ are the Fourier coefficients of $\mathrm{v}$ for basis $P_{\ell}$ and quadrature rule ${Q}_{N_{j-1}}$ on $T^{2}$.

The reconstruction involving the operations of convolution and upsampling is given by the following proposition.

\begin{proposition}
\label{prop:reconstruction}
Let the framelet coefficients for semi-discrete framelets in \eqref{eq:fra} and \eqref{eq:frb} be given by \eqref{eq:coe.framelets}, where the supports of $\widehat{\alpha}$ and $\widehat{\beta^{n}}$ are subsets of $[0,1/2]$, and \eqref{eq:tight.masks} holds.
Then, for $j\ge 1$, the reconstruction from level $j-1$ to level $j$ is 
\begin{equation}\label{eq:reconstr.j.j1}
  \mathrm{v}_{j} =  (\mathrm{v}_{j-1}\hspace{-0.8mm}\uparrow_{j}) \ast_{j} a+\sum_{n=1}^r   \mathrm{w}^{n}_{j-1} \ast_{j} b_{n}.
\end{equation}
\end{proposition}
\begin{proof} 
By Proposition~\ref{prop:decomposition}, for $k=1,\dots,N_{j}$,
\begin{equation*}
((\mathrm{v}_{j-1}\hspace{-0.8mm}\uparrow_{j}) \ast_{j} a)_{k}
= \sqrt{\omega_{j,k}}\sum_{\lambda_{\ell}\le 2^{j-1}}\widehat{(\mathrm{v}_{j})}_{\ell} \left|\widehat{a}\left(\frac{\lambda_{\ell}}{2^j}\right)\right|^2{P_{\ell}}(\boldsymbol{x}_{j,k})
\label{eq:sub.tran.vj.a}
\end{equation*}
and
\begin{equation*}
(\mathrm{w}^{n}_{j-1} \ast_{j} b_{n})_{k}
= \sqrt{\omega_{j,k}}\sum_{\lambda_{\ell}\le 2^{j-1}}\widehat{(\mathrm{v}_{j})}_{\ell} \left|\widehat{b}_{n}\left(\frac{\lambda_{\ell}}{2^j}\right)\right|^2{P_{\ell}}(\boldsymbol{x}_{j,k}),\quad n = 1,\ldots,r.
\label{eq:sub.tran.vj.b}
\end{equation*}
These give
\begin{align*}
&\Bigl((\mathrm{v}_{j-1}\hspace{-0.8mm}\uparrow_{j}) \ast_{j} a+\sum_{n=1}^r   \mathrm{w}^{n}_{j-1} \ast_{j} b_{n}\Bigr)_{k}\\
&\quad= \sqrt{\omega_{j,k}}\sum_{\lambda_{\ell}\le2^{j-1}}\widehat{(\mathrm{v}_{j})}_{\ell}\left( \left|\widehat{a}\left(\frac{\lambda_{\ell}}{2^j}\right)\right|^2+\sum_{n=1}^r
 \left|\widehat{b}_{n}\left(\frac{\lambda_{\ell}}{2^j}\right)\right|^2\right) P_{\ell}(\boldsymbol{x}_{j,k})\\
&\quad=\sqrt{\omega_{j,k}}\sum_{\lambda_{\ell}\le 2^{j-1}}\widehat{(\mathrm{v}_{j})}_{\ell}\: P_{\ell}(\boldsymbol{x}_{j,k})\\
&= (\mathrm{v}_{j})_{k},
\end{align*}
thus proving \eqref{eq:reconstr.j.j1}.
\end{proof}

\begin{remark} 
\cite[Theorem~3.1]{WaZh2017} proves \eqref{eq:reconstr.j.j1} for general Riemannian manifolds when the quadrature rule is exact for polynomials up to degree $2^j$ and under condition \eqref{eq:tight.masks}.	Here we do not require that the quadrature rules of the framelets satisfy the polynomial exactness.
\end{remark}

Repeatedly using the decomposition and reconstruction in Propositions~\ref{prop:decomposition} and~ \ref{prop:reconstruction} gives multi-level framelet transforms. Figure~\ref{fig:ml.fmt} illustrates the decomposition and reconstruction for levels $0,\dots,j$.

\begin{figure}
\begin{minipage}{\textwidth}
\begin{minipage}{0.48\textwidth}
\centering
\includegraphics[scale=0.375]{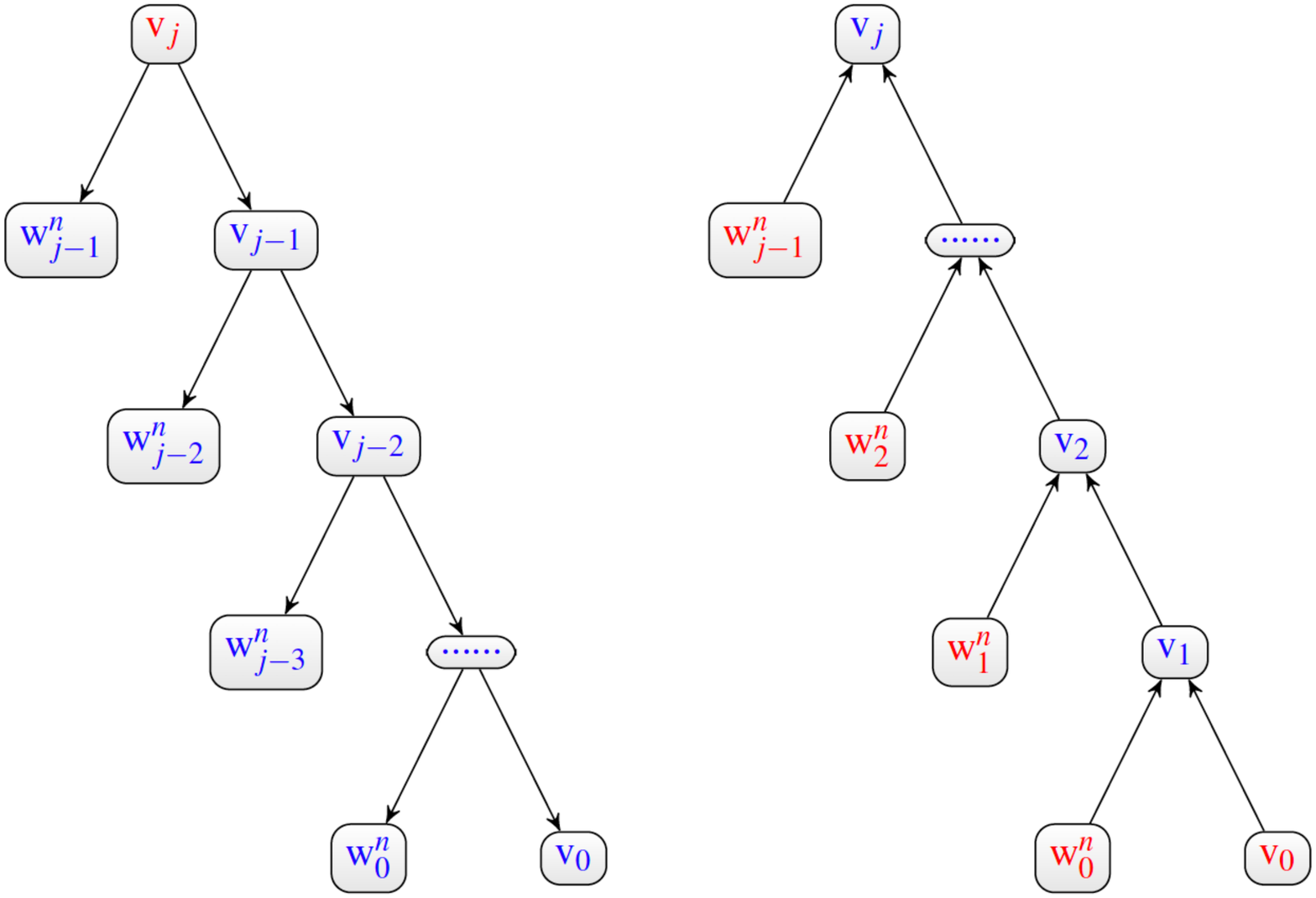}
\end{minipage}
\caption{The left diagram illustrates the decomposition of the framelets coefficients which computes all coefficients in lower levels by $\mathrm{v}_{j}$. The right shows the reconstruction of framelet coefficients $\mathrm{v}_{j}$ from the coefficients $\mathrm{v}_{0}$ and $\mathrm{w}^{n}_{0},\dots,\mathrm{w}^{n}_{j-1}$ of lower levels.}
\label{fig:ml.fmt}
\end{minipage}
\end{figure}

\section{Constructive Examples}\label{sec:constr.example}
From the above analysis, the construction of semi-discrete framelets needs an orthonormal basis for $L_{2}(T^2)$ and appropriate masks and quadrature rules.  

\textbf{Orthonormal bases.} 
One orthonormal basis can be constructed by the tensor product of Jacobi polynomials, see \cite[Proposition~2.4.1]{DuXu2014}. For $\tau,\gamma>-1$ and $\ell\ge0$, let $P^{(\tau,\gamma)}_{\ell}(t)$ be the Jacobi polynomial of degree $\ell$ with respect to the weight $(1-t)^{\tau}(1+t)^{\gamma}$ on [-1,1].
For $\boldsymbol{x}:=(x_{1},x_{2})\in T^{2}$ and $\ell\in\mathbb{N}_{0}$ and $m=0,\dots,\ell$,
let
\begin{equation}\label{eq:eigfx.Jcb}
  P_{\ell,m}(\boldsymbol{x}) := \sqrt{(\ell+1)(2m+1)}\: P^{(2m+1,0)}_{\ell-m}(2x_{1}-1)(1-x_{1})^{m}P^{(0,0)}_{m}\left(\frac{2x_{2}}{1-x_{1}}-1\right).
\end{equation}
Then $\{P_{\ell,m}|\hspace{0.2mm} m=0,\dots,\ell\}$ is an orthonormal basis of $\mathcal{V}_{\ell}$ and $\{P_{\ell,m} |\hspace{0.2mm} m=0,\dots,\ell,\ell\ge0\}$ forms an orthonormal basis of $L_{2}(T^2)$.

Sun \cite{Sun2003} constructs another orthonormal basis for $L_{2}(T^2)$, which is useful in discrete Fourier analysis on $T^{2}$, see \cite{LiSuXu2008,LiXu2010}.

\textbf{Masks.} We give an example of masks with two high-passes.
Let 
\begin{equation*}
  \nu(t) := t^{4}(35 - 84t + 70t^{2} - 20 t^{3}), \quad t\in\mathbb{R}.
\end{equation*}
By \cite[Chapter~4]{Daubechies1992}, the masks $a, b_{1}$ and $b_{2}$ can be defined by their Fourier series as
\begin{align}
  \widehat{a}(\xi)&: =
  \left\{\begin{array}{ll}
    1, & |\xi|<\frac{1}{8},\\[1mm]
    \cos\bigl(\frac{\pi}{2}\hspace{0.3mm} \nu(8|\xi|-1)\bigr), & \frac{1}{8} \le |\xi| \le \frac{1}{4},\\[1mm]
    0, & \frac14<|\xi|\le\frac12,
    \end{array}\right.\label{eq:mask.numer.s3.a} \\[1mm]
  \widehat{b_{1}}(\xi)&:  =\left\{\begin{array}{ll}
    0, & |\xi|<\frac{1}{8},\\[1mm]
    \sin\bigl(\frac{\pi}{2}\hspace{0.3mm} \nu(8|\xi|-1)\bigr), & \frac{1}{8} \le |\xi| \le \frac{1}{4},\\[1mm]
    \cos\bigl(\frac{\pi}{2}\hspace{0.3mm} \nu(4|\xi|-1)\bigr), & \frac14<|\xi|\le\frac12.
    \end{array}\right. \label{eq:mask.numer.s3.b1}\\[1mm]
  \widehat{b_{2}}(\xi)&:
  =\left\{\begin{array}{ll}
    0, & |\xi|<\frac{1}{4},\\[1mm]
    \sin\bigl(\frac{\pi}{2}\hspace{0.3mm} \nu(4|\xi|-1)\bigr), & \frac{1}{4} \le |\xi| \le \frac{1}{2},
    \end{array}\right. \label{eq:mask.numer.s3.b2}
\end{align}
which satisfy \eqref{eq:tight.masks}.

The corresponding scaling functions are
\begin{align}
  \widehat{\alpha}(\xi)&=
  \left\{\begin{array}{ll}
    1, & |\xi|<\frac{1}{4},\\[1mm]
    \cos\bigl(\frac{\pi}{2}\hspace{0.3mm} \nu(4|\xi|-1)\bigr), & \frac{1}{4} \le |\xi| \le \frac{1}{2},\\[1mm]
    0, & \hbox{else},
    \end{array}\right. \label{eq:scala}\\[1mm]
  \widehat{\beta^1}(\xi)&= \left\{\begin{array}{ll}
    \sin\left(\frac{\pi}{2}\hspace{0.3mm} \nu(4|\xi|-1)\right), & \frac{1}{4}\le|\xi|<\frac{1}{2},\\[1mm]
    \cos^2\left(\frac{\pi}{2}\hspace{0.3mm} \nu(2|\xi|-1)\right), & \frac{1}{2} \le |\xi| \le 1,\\[1mm]
    0, & \hbox{else},
    \end{array}\right. \label{eq:scalb1}\\[1mm]
  \widehat{\beta^2}(\xi)&= \left\{\begin{array}{ll}
   0, &|\xi|<\frac{1}{2},\\[1mm]
    \cos\left(\frac{\pi}{2}\hspace{0.3mm} \nu(2|\xi|-1)\right) \sin\left(\frac{\pi}{2}\hspace{0.3mm} \nu(2|\xi|-1)\right), & \frac{1}{2} \le |\xi| \le 1,\\[1mm]
    0, & \hbox{else}.
    \end{array}\right. \label{eq:scalb2}
\end{align}
Here, $\mathrm{supp}\:\widehat{\alpha}\subseteq [0,1/2]$ and $\mathrm{supp}\:\widehat{\beta^{n}}\subseteq [1/4,1]$, $n=1,2$. This means that the scaling of the framelet $\boldsymbol{\varphi}_{j,k}$ in \eqref{eq:fra} with the low-pass scaling function in \eqref{eq:scala} is half of the scaling of the framelets $\boldsymbol{\psi}_{j,k}^{1}$ and $\boldsymbol{\psi}_{j,k}^{2}$ in \eqref{eq:frb} with high-pass scaling functions in \eqref{eq:scalb1} and \eqref{eq:scalb2}. The high-pass framelets thus need to use a quadrature rule at the level $j+1$, one level higher than $\boldsymbol{\varphi}_{j,k}$. 

Figure~\ref{fig:masks} shows the Fourier series of masks $a$, $b_{1}$ and $b_{2}$ in \eqref{eq:mask.numer.s3.a}, \eqref{eq:mask.numer.s3.b1} and \eqref{eq:mask.numer.s3.b2}.

\begin{figure}
\vspace{1mm}
    \centering
    \includegraphics[scale=0.45]{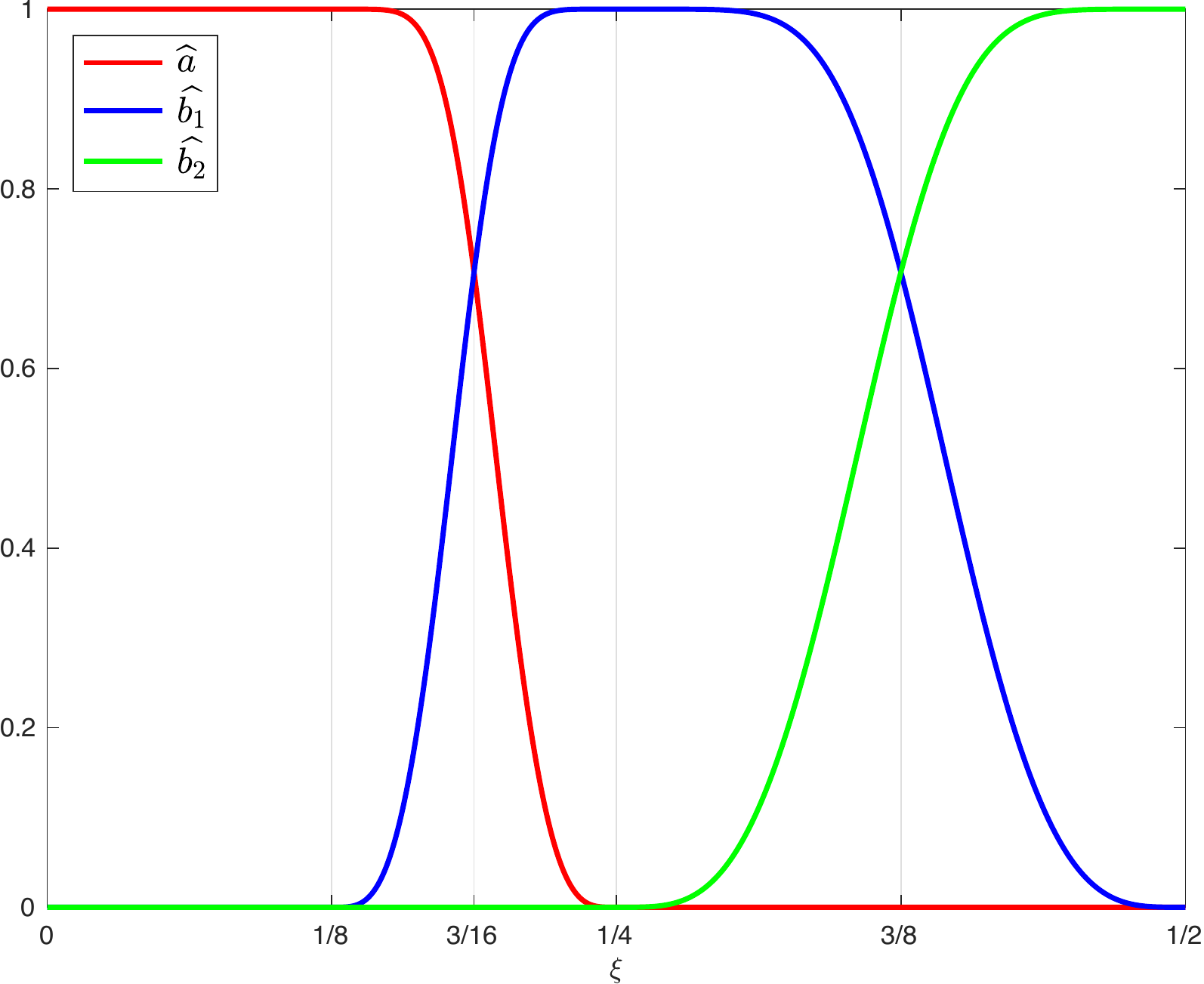}
    \caption{The red curve shows the Fourier series of the low-pass mask $\widehat{a}$ in \eqref{eq:mask.numer.s3.a} which has support in $[0,1/4]$. The blue and green curves show the Fourier series of high-pass masks $\widehat{b_{1}}$ and $\widehat{b_{2}}$ in \eqref{eq:mask.numer.s3.b1} and \eqref{eq:mask.numer.s3.b2} whose supports are subsets of $[0,1/2]$.}\label{fig:masks}
\end{figure}

\begin{figure}
\vspace{3mm}
    \centering
    \includegraphics[scale=0.45]{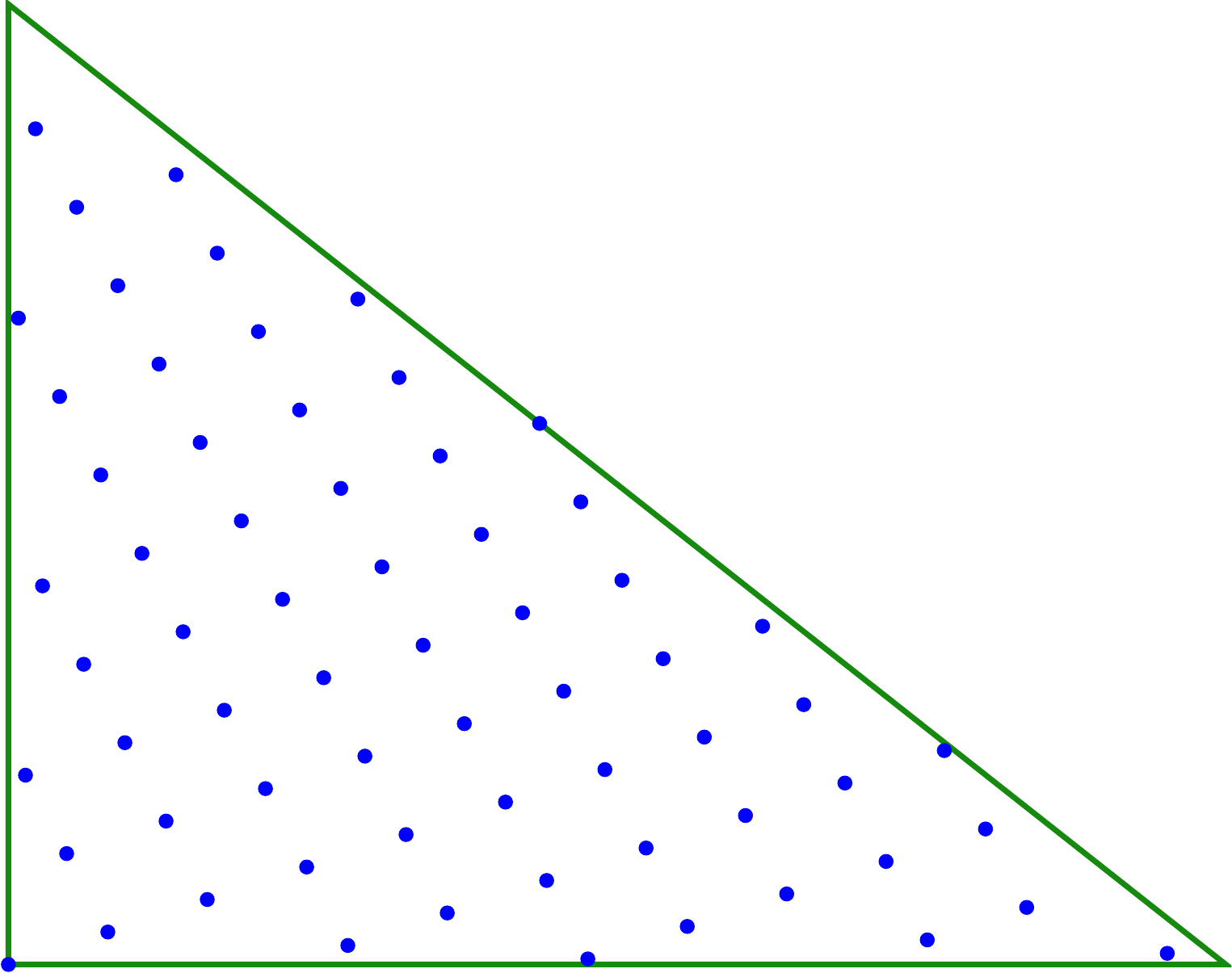}\vspace{2mm}
    \caption{Triangular Kronecker lattice with $65$ nodes for framelets $\boldsymbol{\varphi}_{3,k}$ and $\boldsymbol{\psi}_{2,k}^{n}$.}\label{fig:tri.pnt.N65}
\end{figure}

\textbf{Quadrature rules.}  We use triangular Kronecker lattices of Basu and Owen \cite{BaOw2015} with equal weights as the quadrature rules for framelets, which are shifted lattice points intersecting with the simplex. For the quadrature rule ${Q}_{N_{j}}$ of framelets, we use the triangular Kronecker lattice with at least $2^{2j}$ nodes, which are the translation points of the low-pass framelets $\boldsymbol{\varphi}_{j,k}$ at level $j$ and those of high-pass framelets $\boldsymbol{\psi}_{j-1,k'}^{n}$ at level $j-1$.  Figure~\ref{fig:tri.pnt.N65} shows the triangular Kronecker lattice with $N=65$ nodes on $T^{2}$ used for framelets at levels $2$ and $3$.

\textbf{Framelets.} Using the orthonormal basis in \eqref{eq:eigfx.Jcb}, scaling functions in \eqref{eq:scala}--\eqref{eq:scalb2} and triangular Kronecker lattices with equal weights, the framelets are, for $j\in\mathbb{N}_{0}$, 
\begin{equation}\label{eq:fra.Jcb}
    \boldsymbol{\varphi}_{j,k}(\boldsymbol{x}) = \frac{1}{\sqrt{N_{j}}}\sum_{\ell=0}^{\infty}\sum_{m=0}^{\ell}\widehat{\alpha}\left(\frac{\sqrt{\ell(\ell+2)}}{2^{j}}\right)\overline{P_{\ell,m}(\boldsymbol{x}_{j,k})}P_{\ell,m}(\boldsymbol{x}),\quad k=1,\dots,N_{j},
\end{equation}
and for $n =1,2$,
\begin{equation}\label{eq:frb.Jcb}
    \boldsymbol{\psi}_{j,k'}^{n}(\boldsymbol{x}) =  \frac{1}{\sqrt{N_{j+1}}}\sum_{\ell=0}^{\infty}\sum_{m=0}^{\ell} \widehat{\beta^{n}}\left(\frac{\sqrt{\ell(\ell+2)}}{2^{j}}\right)\overline{P_{\ell,m}(\boldsymbol{x}_{j+1,k'})}P_{\ell,m}(\boldsymbol{x}),\quad k'=1,\dots,N_{j+1}.
\end{equation}

Figure~\ref{fig:fr} shows the framelets $\boldsymbol{\varphi}_{j,k}$, $\boldsymbol{\psi}_{j,k'}^{1}$ and $\boldsymbol{\psi}_{j,k'}^{2}$ at level $j=5$ with $k=512$ and $k'=2048$, using orthonormal basis \eqref{eq:eigfx.Jcb} and scaling functions \eqref{eq:scala}, \eqref{eq:scalb1} and \eqref{eq:scalb2}, translated at the triangular Kronecker lattice points $\boldsymbol{x}_{5,512}$, $\boldsymbol{x}_{6,2048}$ and $\boldsymbol{x}_{6,2048}$. The total number of low-pass framelets $\boldsymbol{\varphi}_{j,k}$ at level $j=5$ is $N_5=1025$ and the total number of high-pass framelets $\boldsymbol{\psi}_{j,k'}^{n}$, $n=1$ or $2$, at level $j=5$ is $N_{6}=4097$. The pictures show that the framelets $\boldsymbol{\varphi}_{5,512}$, $\boldsymbol{\psi}_{5,2048}^{1}$ and $\boldsymbol{\psi}_{5,2048}^{2}$ are radial functions on $T^{2}$ with centers at the translation points $\boldsymbol{x}_{5,512}$, $\boldsymbol{x}_{6,2048}$ and $\boldsymbol{x}_{6,2048}$ respectively.

We observe that the high-pass framelets $\boldsymbol{\psi}_{5,2048}^{1}$ and $\boldsymbol{\psi}_{5,2048}^{2}$ are highly concentrated at the translation point $\boldsymbol{x}_{6,2048}$, and are more localised than the low-pass framelet at the same level. This illustrates that the high-pass framelets can be used to depict details of a function on $T^{2}$ in multiresolution analysis.

\begin{figure}
  \begin{minipage}{\textwidth}
  \begin{minipage}{0.49\textwidth}
  \centering
  \includegraphics[trim = 0mm 0mm 0mm 0mm, width=1.1\textwidth]{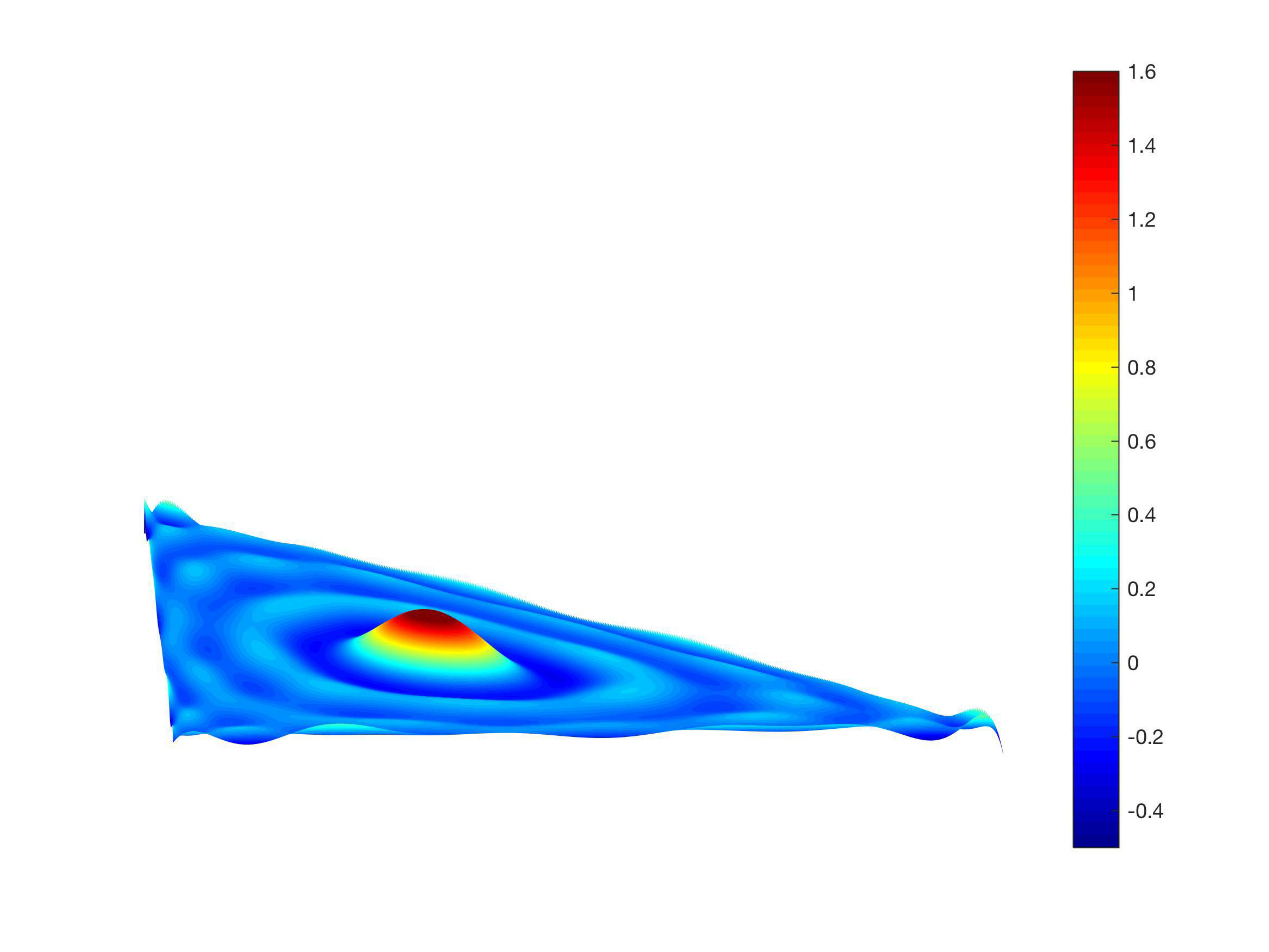}\\[-6mm]
  {$\boldsymbol{\varphi}_{5,512}$~~~~~}\label{fig:fra}
  \end{minipage}\\
  \begin{minipage}{0.49\textwidth}
  \centering
  \includegraphics[trim = 0mm 0mm 0mm 0mm, width=1.1\textwidth]{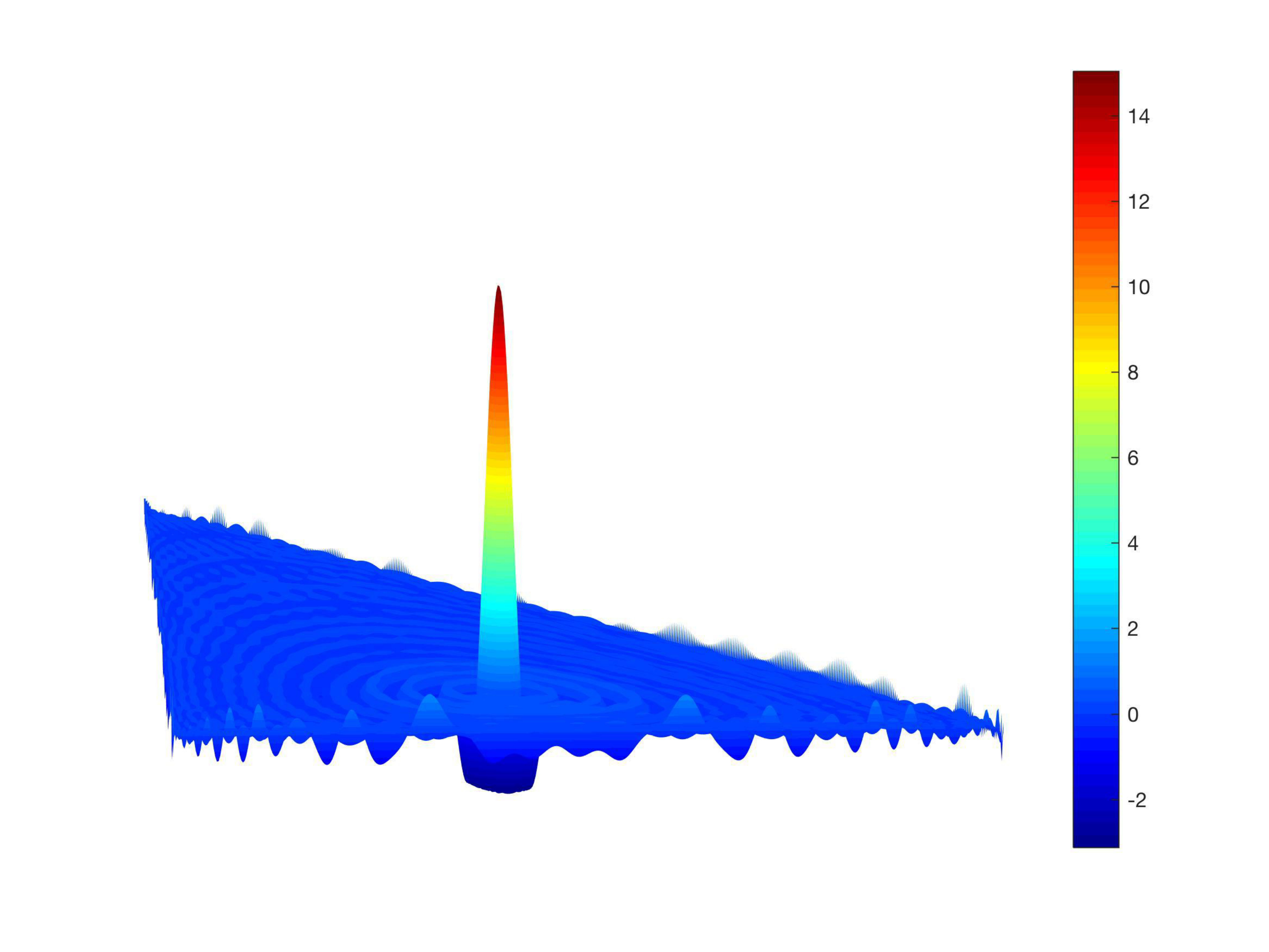}\\[-4mm]
  {$\boldsymbol{\psi}_{5,2048}^{1}$~~~~~}\label{fig:frb1}
  \end{minipage}
  \begin{minipage}{0.49\textwidth}
  \centering
  \includegraphics[trim = 0mm 0mm 0mm 0mm, width=1.1\textwidth]{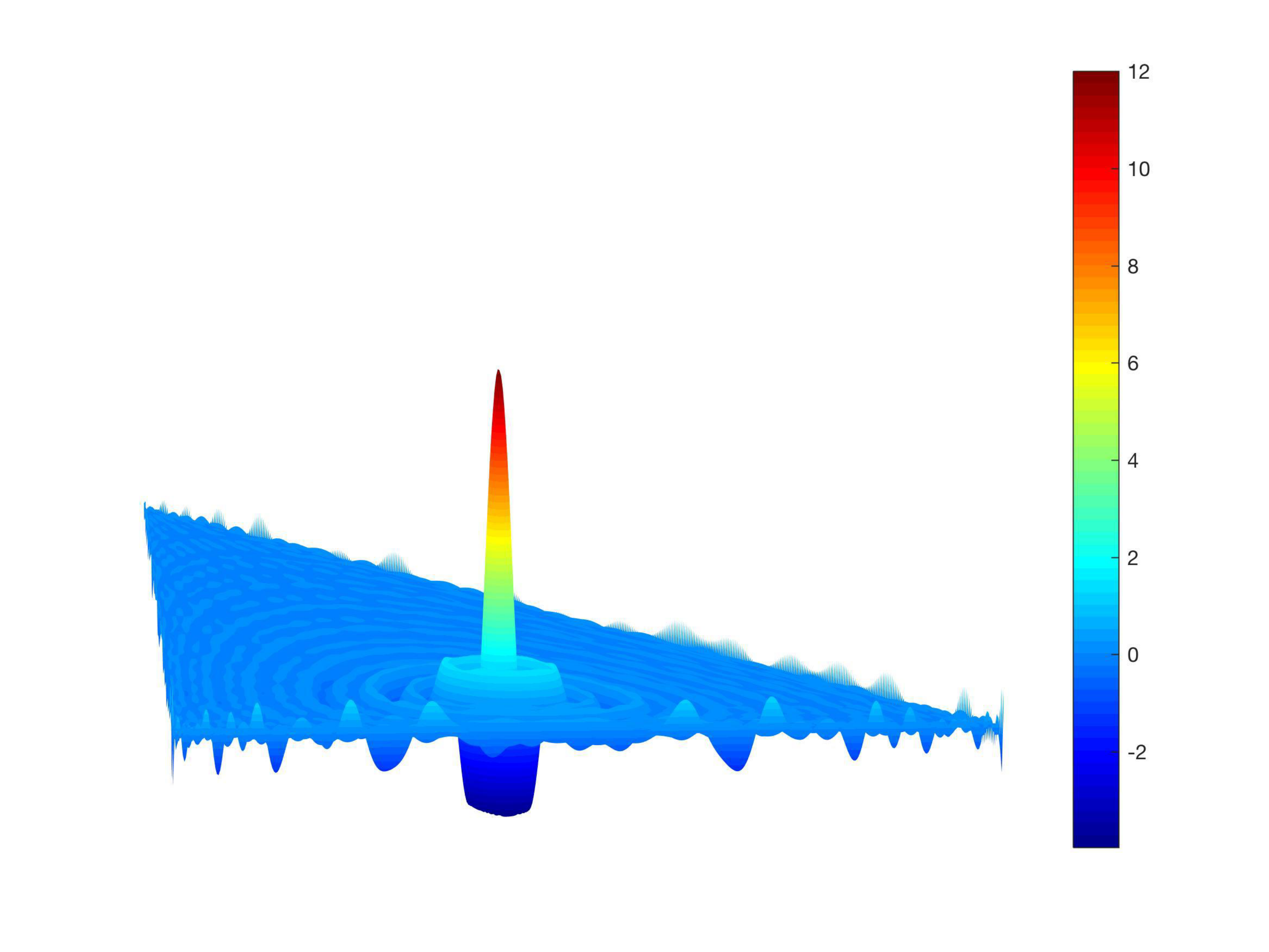}\\[-4mm]
  {$\boldsymbol{\psi}_{5,2048}^{2}$}\label{fig:frb2}
  \end{minipage}
  \end{minipage}\\[1mm]
  \begin{minipage}{\textwidth}
\caption{\scriptsize The three pictures show framelets $\boldsymbol{\varphi}_{5,512}$, $\boldsymbol{\psi}_{5,2048}^{1}$ and $\boldsymbol{\psi}_{5,2048}^{2}$ given by \eqref{eq:fra.Jcb} and \eqref{eq:frb.Jcb} at level $j=5$.
}\label{fig:fr}
\end{minipage}
\end{figure}

\section{Fast Computing}\label{sec:fast.comput}
The framelet transforms on $T^{2}$ can be represented by discrete Fourier transforms on the simplex. This implies a fast computational strategy of the decomposition and reconstruction for framelets.

We use the notation of Sections~\ref{sec:decomposition} and \ref{sec:reconstruction}. Let $j\in\mathbb{N}_{0}$ and let $\Lambda_j$ be the largest integer $\ell$ such that $\lambda_{\ell}\le 2^{j-1}$.
The \emph{discrete Fourier transform} (DFT) for a sequence $\mathrm{u}\in l(\Lambda_j)$ is the sequence $\mathbf{F}_{j} \mathrm{u}$ in $l(N_j)$ such that
\begin{equation}\label{eq:fft}
(\mathbf{F}_{j} \mathrm{u})_{k} :=   \sum_{\ell=0}^{\Lambda_j}\mathrm{u}_{\ell}\: \sqrt{\omega_{j,k}}\:P_{\ell}(\boldsymbol{x}_{j,k}),\quad k = 0,\ldots, N_{j}.
\end{equation}
The \emph{adjoint discrete Fourier transform} (adjoint DFT) $\mathbf{F}_{j}^*$ of a sequence $\mathrm{v}\in l(N_j)$ is the sequence $\mathbf{F}_{j}^*\mathrm{v}$ in $l(\Lambda_j)$ such that
\begin{equation}\label{eq:adjfft}
(\mathbf{F}_{j}^*\mathrm{v})_\ell :=    \sum_{k=0}^{N_{j}}  {\mathrm{v}}_{k}\: \sqrt{\omega_{j,k}}\:\overline{P_{\ell}(\boldsymbol{x}_{j,k})}, \quad \ell=0,\dots,\Lambda_j.
\end{equation}
The DFTs on the simplex in \eqref{eq:fft} and \eqref{eq:adjfft} using the orthonormal basis $P_{\ell}$ are the analogues of DFTs for square-integrable periodic functions on $\mathbb{R}$ which use the orthogonal basis $e^{2\pi{\mathrm{i}} \ell' x}$, $\ell'\in\mathbb{Z}$.

By \eqref{eq:fft} and \eqref{eq:adjfft}, we can rewrite the decomposition in \eqref{eq:decomp} and reconstruction in \eqref{eq:reconstr.j.j1} as
\[
  \mathrm{v}_{j-1} = \mathbf{F}_{j-1}(\widehat{\mathrm{v}_{j}\ast_{j} a^\star}),\quad \mathrm{w}^{n}_{j-1} = \mathbf{F}_{j}(\widehat{\mathrm{v}_{j}\ast_{j} {(b_{n})}^\star}),\quad n = 1,\ldots, r
\]
and
\[
  \mathrm{v}_{j} = \left(\mathbf{F}_{j}^*(\mathrm{v}_{j-1})\right) \ast_{j} a + \sum_{n=1}^r \left(\mathbf{F}_{j}^*(\mathrm{w}^{n}_{j-1})\right) \ast_{j} b_{n}.
\]
This means that the decomposition from level $j$ to level $j-1$ is the DFTs of convolutions of the level-$j$ framelet coefficients with masks, and that the reconstruction from level $j-1$ to level $j$ is the sum of convolutions of the adjoint DFTs of level-$(j-1)$ coefficients with masks. Since the convolution is the sum of point-wise multiplications, the computational steps of the framelet transforms are in proportion to those of DFTs on the simplex. 

The FFT on $T^{2}$ uses, up to log factors, $\mathcal{O}(N)$ operations for an input sequence of size $N$.
If for $j\ge1$, the ratio $N_{j}/N_{j-1}$ of the numbers of the nodes of the quadrature rules ${Q}_{N_{j}}$ and ${Q}_{N_{j-1}}$ is equivalent to a constant $C$, $C>1$, the computational steps of the framelet transforms (both the decomposition and reconstruction) between levels $0,1,\dots,{J}$, ${J}\ge1$, are $\mathcal{O}((r+1)N_J)$ for the sequence $\mathrm{v}_{J}$ of the framelet coefficients of size $N_{{J}}$, and the redundancy rate of the framelet transforms is also $\mathcal{O}((r+1)N_J)$. The framelets with the quadrature rules using triangular Kronecker lattices, as shown in Section~\ref{sec:constr.example}, satisfy that $N_{j}/N_{j-1}\sim 4$. Thus, the framelet transforms between levels $0$ to $J$ have  computational steps in proportion to $2^{2J}$.


\section{Discussion}
In the paper, we only consider the framelet transforms for one framelet system with starting level $0$. The results can be generalized to a sequence of framelet systems as \cite{WaZh2017,Han2010}, which will allow one more flexibility in constructing framelets.

The decomposition holds for framelets with any quadrature rules on the simplex. In order to achieve the tightness of the framelets and thus exact reconstruction for functions on the simplex by framelets, the quadrature rules are required to be exact for polynomials,
see Sections~\ref{sec:decomposition} and \ref{sec:reconstruction}.
However, polynomial-exact rules are generally difficult to construct on the simplex, see \cite[Chapter~3]{DuXu2014}.

Triangular Kronecker lattices with equal weights used in Section~\ref{sec:constr.example} are low-discrepancy \cite{BaOw2015}, but not exact for polynomials. In this case, the reconstruction will incur errors. 
To overcome this, the masks and quadrature rules shall be constructed to satisfy the condition
\begin{equation*}
\overline{\widehat{a}\left(\frac{\lambda_{\ell}}{2^{j}}\right)}{\widehat{a}}\left(\frac{\lambda_{\ell'}}{2^{j}}\right)\mathcal{U}_{\ell,\ell'}({Q}_{N_{j-1}})
+\sum_{n=1}^r\overline{\widehat{b_{n}}\left(\frac{\lambda_{\ell}}{2^{j}}\right)}{\widehat{b_{n}}}\left(\frac{\lambda_{\ell'}}{2^{j}}\right)\mathcal{U}_{\ell,\ell'}({Q}_{N_{j}}) = \mathcal{U}_{\ell,\ell'}({Q}_{N_{j}}), \label{thmeq:fr.tightness.a.mask.cond}
 \end{equation*}
 for $j\ge 1$ and for $\ell,\ell'\in\N_0$ satisfying $\overline{\widehat{\alpha}\left(\frac{\lambda_{\ell}}{2^{j}}\right)}{\widehat{\alpha}}\left(\frac{\lambda_{\ell'}}{2^{j}}\right)\neq0$, 
 where 
  $ \mathcal{U}_{\ell,\ell'}({Q}_{N_{j}}):=\sum_{k=0}^{N_j}\omega_{j,k}P_{\ell}(\boldsymbol{x}_{j,k})\overline{P_{\ell'}(\boldsymbol{x}_{j,k})}$ is the numerical integration of $P_{\ell}\overline{P_{\ell'}}$ over $T^{2}$ by quadrature rule ${Q}_{N_{j}}$,
 see \cite[Theorem~2.4]{WaZh2017}. 
 This condition requires that the quadrature rules for framelets have good properties for numerical integration over the simplex.
 Besides the triangular Kronecker lattices used in the paper, one may consider other quadrature rules with low discrepancy on the simplex, for example, the analogues to quasi-Monte Carlo (QMC) points in the cube and spheres, see \cite{BrSaSlWo2014,DiKuSl2013,SlJo1994}.

To implement the fast algorithms for the DFTs in \eqref{eq:fft} and \eqref{eq:adjfft}, we need fast transforms for the bases $P_{\ell}$. For example, we can represent the bases $P_{\ell,m}$ in \eqref{eq:eigfx.Jcb} by trigonometric polynomials and apply the FFT on $\mathbb{R}$ to achieve fast algorithms for the DFTs on $T^{2}$.

\begin{acknowledgement}
The authors thank the anonymous referees for their valuable comments.
We are grateful to Kinjal Basu, Yuan Xu and Xiaosheng Zhuang for their helpful discussions.
\end{acknowledgement}


\end{document}